\documentclass[11pt]{amsart}
\usepackage{amsmath,amssymb,amscd, amsthm}
\usepackage{epic,eepic,epsfig}
\usepackage{comment} 
\usepackage{mathrsfs}
\usepackage[all,cmtip]{xy}
\usepackage{graphicx}
\usepackage{here}

\makeatletter
\def\@settitle{\begin{center}%
  \baselineskip14\p@\relax
  \normalfont\LARGE\bfseries
  \@title
  \ifx\@subtitle\@empty\else
     \\[1ex] 
     \normalsize\mdseries\@subtitle
  \fi
 \ifx\@didication\@empty\else
     \\[2ex] 
     \large\mdseries\it\@dedication
  \fi
  \end{center}%
}
\def\subtitle#1{\gdef\@subtitle{#1}}
\def\@subtitle{}
\def\dedication#1{\gdef\@dedication{#1}}
\def\@dedication{}

\renewcommand{\section}{\@startsection
{section}{1}{0mm}{5mm}{2mm}{\raggedright\bfseries}}
 
  \@addtoreset{equation}{section}
\makeatother
\newtheorem{mainthm}{Theorem} 

\newtheorem{maincor}[mainthm]{Corollary} 
\newtheorem{theorem}{Theorem}[section] 

\theoremstyle{definition}

\newtheorem{thm}[theorem]{Theorem}
\newtheorem{lem}[theorem]{Lemma}
\newtheorem{prop}[theorem]{Proposition}
\newtheorem{cor}[theorem]{Corollary}
\newtheorem{dfn}[theorem]{Definition}

\newtheorem*{Example*}{Example}
\newtheorem*{Claim*}{Claim}
\newtheorem*{Question*}{\it Question}

\begin{document}

\def\check{{\clubsuit}}
\def\Z{{\mathbb Z}}
\def\G{{\mathbb G}}
\def\C{{\mathbb C}}
\def\Q{{\mathbb Q}}
\def\R{{\mathbb R}}
\def\N{{\mathbb N}}
\def\bQ{\overline{\mathbb Q}}
\def\Gal{\mathrm{Gal}}
\def\Out{\mathrm{Out}}
\def\vru{\,\vrule\,}
\def\wf{{f}}
\def\ff{\mathfrak{f}}
\def\et{\text{\'et}}
\def\check{{\clubsuit}}
\def\boldzeta{{\boldsymbol \zeta}}
\def\nyoroto{{\rightsquigarrow}}
\newcommand{\pathto}[3]{#1\overset{#2}{\dashto} #3}
\newcommand{\pathtoD}[3]{#1\overset{#2}{-\dashto} #3}
\def\dashto{{\,\!\dasharrow\!\,}}
\def\ovec#1{\overrightarrow{#1}}
\def\Isom{\mathrm{Isom}}
\def\proP{{\text{pro-}p}}
\def\padic{{p\mathchar`-\mathrm{adic}}}
\def\scP{\mathscr{P}}
\def\scM{\mathscr{M}}
\def\scLi{{\mathscr{L}i}}
\newcommand{\Li}{\mathrm{Li}}
\newcommand{\cC}{\mathcal{C}}
\def\tilbchi{\tilde{\boldsymbol \chi}}
\def\tilchi{{\tilde{\chi}}}
\def\bkappa{{\boldsymbol \kappa}}
\def\lala{\la\!\la}
\def\rara{\ra\!\ra}
\def\ttx{{\mathtt{x}}}
\def\tty{{\mathtt{y}}}
\def\ttz{{\mathtt{z}}}
\def\kk{{\varkappa}}     
\def\la{{\langle}}
\def\ra{{\rangle}}
\def\fp{{\mathfrak{p}}}

\newcommand{\nc}{\newcommand}
\nc{\hra}{\hookrightarrow}
\nc{\ab}{\mathrm{ab}}
\nc{\id}{\mathrm{id}}
\nc{\defeq}{:=}
\nc{\sseq}{\subset}
\nc{\epi}{\twoheadrightarrow}
\nc{\isom}{\stackrel{\sim}{\to}}
\nc{\Ker}{\mathrm{ker}}
\nc{\kb}{\overline{k}}
\nc{\spec}{\mathrm{Spec}}
\nc{\aut}{\mathrm{Aut}}
\nc{\out}{\mathrm{Out}}
\nc{\im}{\mathrm{Im}}
\nc{\inn}{\mathrm{Inn}}
\nc{\Sr}{\mathfrak{S}_r}
\nc{\pri}{\mathrm{pr}_i}
\nc{\cc}{\mathbb{C}}
\nc{\nn}{\mathbb{N}}
\nc{\rr}{\mathbb{R}}
\nc{\zz}{\mathbb{Z}}
\nc{\rrp}{\rr_{>0}}
\nc{\rrnn}{\rr_{\ge0}}
\nc{\qq}{\mathbb{Q}}
\nc{\fl}{\mathbb{F}_l}
\nc{\qbar}{\overline{\qq}}
\nc{\po}{\mathbb{P}^1}
\nc{\St}{\mathfrak{S}_3}
\nc{\Sf}{\mathfrak{S}_4}
\nc{\mbn}{\mathcal{B}_n}
\nc{\mbf}{\mathcal{B}_4}
\nc{\mbt}{\mathcal{B}_3}
\nc{\mbnh}{\widehat{\mathcal{B}}_n}
\nc{\mpn}{\mathcal{P}_n}
\nc{\mpnh}{\widehat{\mathcal{P}}_n}
\nc{\mpth}{\widehat{\mathcal{P}}_3}
\nc{\sn}{\mathfrak{S}_n}
\nc{\st}{\mathfrak{S}_3}
\nc{\stab}{\mathrm{Stab}}
\nc{\bn}{B_n}
\nc{\pn}{P_n}
\nc{\bnh}{\widehat{B}_n}
\nc{\pnh}{\widehat{P}_n}
\nc{\pin}{\varpi_n}
\nc{\pinh}{\widehat{\pi}_n}
\nc{\mpin}{\pi_n}
\nc{\mpinh}{\widehat{\pi}_n}
\nc{\mpith}{\widehat{\pi}_3}
\nc{\mpifh}{\widehat{\pi}_4}
\nc{\cn}{C_n}
\nc{\cnh}{\widehat{C}_n}
\nc{\gn}{\Gamma_{0, [n]}}
\nc{\pgn}{\Gamma_{0, n}}
\nc{\sno}{\mathfrak{S}_{n+1}}
\nc{\gno}{\Gamma_{0, [n+1]}}
\nc{\gnoh}{\widehat{\Gamma}_{0, [n+1]}}
\nc{\pgno}{\Gamma_{0, n+1}}
\nc{\pgnoh}{\widehat{\Gamma}_{0, n+1}}
\nc{\pgtoh}{\widehat{\Gamma}_{0, 4}}
\nc{\pzfh}{\widehat{\Pi}_{0, 4}}
\nc{\mzno}{\mathcal{M}_{0, n+1}}
\nc{\phig}{\phi_{\Gamma}}
\nc{\phip}{\phi_{\mathcal{P}}}
\nc{\phis}{\phi_{\mathfrak{S}}}
\nc{\gt}{\widehat{\mathrm{GT}}}
\nc{\mbfh}{\widehat{\mathcal{B}}_4}
\nc{\mpfh}{\widehat{\mathcal{P}}_4}
\nc{\mbth}{\widehat{\mathcal{B}}_3}
\nc{\zhat}{\widehat{\zz}}
\nc{\co}{H_{\mathrm{cont}}}
\nc{\coz}{Z_{\mathrm{cont}}}
\nc{\cob}{B_{\mathrm{cont}}}
\nc{\frakS}{\mathfrak{S}}
\nc{\cB}{\mathcal{B}}
\nc{\cP}{\mathcal{P}}
\nc{\freeprod}{\,\rotatebox[origin = c]{180}{\mbox{\footnotesize$\Pi$}}\,}
\nc{\hookuparrow}{\rotatebox[origin = c]{90}{\mbox{$\hookrightarrow$}}}
\nc{\mathj}{\jmath}
\nc{\zh}{\widehat{\zz}}
\nc{\vertsim}{\reflectbox{\rotatebox{90}{$\sim$}}}

\title{The automorphism groups \\
of the profinite braid groups}

\author{Arata Minamide and Hiroaki Nakamura}


\address{Arata Minamide:
University of Nottingham,
School of Mathematical Sciences,
University Park
Nottingham NG7 2RD, United Kingdom;
Mathematical Institute,
University of Oxford,
Woodstock Road
Oxford OX2 6GG, United Kingdom}
\email{arata373de@gmail.com, minamide@kurims.kyoto-u.ac.jp} 

\address{Hiroaki Nakamura: 
Department of Mathematics, 
Graduate School of Science, 
Osaka University, 
Toyonaka, Osaka 560-0043, Japan}
\email{nakamura@math.sci.osaka-u.ac.jp}

\maketitle

\markboth{A.Minamide and H.Nakamura}
{The automorphism groups of the profinite braid groups}
\begin{abstract}
In this paper we determine the 
automorphism groups of the profinite braid groups
with four or more strings
in terms of the profinite Grothendieck-Teichm\"uller group.
\end{abstract}

\setcounter{tocdepth}{1}
\tableofcontents



\vspace{-1cm}

\section{Introduction} 

Let $\bn$ be the Artin braid group 
with $n(\ge 2)$ strings defined by 
generators 
$\sigma_1, \sigma_2, \dotso, \sigma_{n-1}$
and relations:
\begin{itemize}
\item \ $\sigma_i \sigma_{i+1}  \sigma_i 
\ = \ \sigma_{i+1}  \sigma_i  \sigma_{i+1}$
$\quad (i=1,\dots,n-1)$,
\item \ $\sigma_i \sigma_j \ = \ \sigma_j  \sigma_i \ \ \ 
(|i-j|\ge2)$.
\end{itemize}
In \cite{dg}, J. L. Dyer and E. K. Grossman studied 
the automorphism group
$\aut(B_n)$ and showed
$\out(B_n)\cong\Z/2\Z$ for $n\ge 3$.
In this paper, we study the continuous automorphisms of the
profinite completion $\bnh$ of $B_n$. We prove

\begin{mainthm} Let $n\ge 4$.
There exists a natural isomorphism 
$$
\out(\widehat B_n) \ \cong \ \gt\times (1+n(n-1)\zh)^{\times},
$$
where 
$\gt$ is the profinite Grothendieck-Teichm\"uller group 
introduced by V.\,Drinfeld \cite{Dr}, Y.\,Ihara \cite{I90}-\cite{im}
and $(1+n(n-1)\zh)^{\times}$ is the kernel of the natural 
projection $\zh^\times\to (\Z/n(n-1)\Z)^\times$.
\end{mainthm}

It is well known that the center $\cnh$ of $\bnh$ is (topologically)
generated by 
$
\zeta_n \defeq (\sigma_1\sigma_2 \dotsm \sigma_{n-1})^n
$ 
and is isomorphic to $\zh$. 
Write
\[
\mbnh \ \defeq \ \bnh/\cnh. 
\]
Since $\cnh$ is a characteristic subgroup of $\bnh$, there is 
induced the natural homomorphism $\gt \to \out(\mbnh)$.
The key fact for the proof of Theorem A is 
the following isomorphism theorem.

\begin{mainthm}[Theorem \ref{gtmbnh}]
Let $n\ge4$. Then, it holds that $\gt \isom \out(\mbnh)$.
\end{mainthm}

Our proofs of Theorems A and B rely on preceding works by many
authors on the Grothendieck-Teichm\"uller group $\gt$
and the profinite completion $\widehat\Gamma_{0,n}$ of the
mapping class group $\Gamma_{0,n}$ of the sphere 
with $n$ marked points (cf.\,\cite{im}, \cite{LS1}, \cite{LS2}, \cite{C12}).
The permutation of labels defines a natural inclusion
of the symmetric group of degree $n$:
$\mathfrak{S}_n\hookrightarrow \Out(\widehat\Gamma_{0,n})$,
whose image commutes with the standard
action of 
$\gt$ 
on $\widehat\Gamma_{0,n}$ (\cite{im}).
D.Harbater and L.Schneps \cite{HS} remarkably showed that 
when $n\ge 5$, $\gt$ is characterized as a ``special'' subgroup of 
the centralizer of $\mathfrak{S}_n$ in $\Out(\widehat\Gamma_{0,n})$.
In a recent work \cite{hmm}, 
this result has been improved by showing 
that the focused centralizer is indeed {\it full} 
as large as possible in $\Out(\widehat\Gamma_{0,n})$.
In particular,

\begin{thm}[Hoshi-Minamide-Mochizuki \cite{hmm} Corollary C]
\label{hmm}
There is a natural 
isomorphism of profinite groups
\[
\gt \times \sno \ \isom \ \out(\pgnoh)
\]
for every integer $n\ge 4$.
\hfill $\square$ 
\end{thm}

Theorems A and B will be derived by translating
the ingredient of Theorem \ref{hmm} for $\Out(\pgnoh)$
into the language of $\Out(\mbnh)$ and $\Out(\bnh)$. 
Arguments given by Dyer-Grossman \cite{dg} for discrete
braid groups generically guide us also in profinite context. 
However, for the case $n=4$, we elaborate a different treatment in \S 3 
due to the existence of non-standard surjections $B_4\epi\frakS_4$ found in E. Artin's classic \cite{a}.
Our argument in \S 3 looks at the ``Cardano-Ferrari''
homomorphism $B_4\epi B_3$ which has close relations with
the universal monodromy representation in once-punctured
elliptic curves. Noting that $\cB_4$ is isomorphic to
the mapping class group $\Gamma_{1,2}$ of a topological torus
with two marked points, we obtain from Theorem B
the following remarkable

\begin{maincor}
There is a natural isomorphism $\gt\isom\out(\widehat\Gamma_{1,2})$.
\qed
\end{maincor}

{\it Acknowledgement}:
The first author would like to thank Prof.\,Shinichi Mochizuki for
helpful discussions and warm encouragements.
During preparation of this manuscript, the authors learnt that
Yuichiro Hoshi and Seidai Yasuda also had discussions on topics
including a similar phase to this paper.
This work was supported by the Research Institute for Mathematical Sciences, 
an International Joint Usage/Research Center located in Kyoto University.
Work on this paper was partially supported by EPSRC programme grant ``Symmetries and Correspondences''
EP/M024830.

\section{Generalities on braid groups}

We begin with recalling basic facts on braid groups (cf.\,e.g.,\,\cite{kt}).
Let $n\ge3$ be an integer.
The pure braid group $\pn$ is the kernel of the epimorphism
\begin{align*}
\varpi_n : \bn & \ \epi \ \sn \\
\sigma_i & \ \mapsto \ (i, i+1)
\qquad (i=1,\dots,n-1).
\end{align*}
The center $C_n$
of $P_n$ coincides with the center of $B_n$ which is a free
cyclic group generated by 
$$
\zeta_n \ \defeq \ (\sigma_1 \sigma_2 \dotsm \sigma_{n-1})^n.
$$
Write $\cP_n:=P_n/C_n$ and $\cB_n:=B_n/C_n$.
The above $\varpi_n$ factors through $\pi_n:\cB_n\epi\frakS_n$
and there arise the following exact sequences of finitely 
generated groups:
\begin{gather}
\label{exactSn}
\begin{CD}
1 @>>> \mpn @>>> \mbn @>\pi_n>> \sn @>>> 1,
\end{CD}
\\
\label{exactBn}
\begin{CD}
1 @>>> \cn @>>> \bn @>>> \mbn @>>> 1,
\end{CD}
\\
\label{exactPn}
\begin{CD}
1 @>>> \cn @>>> \pn @>>> \mpn @>>> 1.
\end{CD}
\end{gather}
We introduce the {\it mapping class group of the $n$-times punctured sphere} $\gn$
to be  the group generated by 
$\bar\sigma_1, \bar\sigma_2, \dotso, \bar\sigma_{n-1}$
with the relations
\begin{itemize}
\item \ $\bar\sigma_i  \bar\sigma_{i+1}  \bar\sigma_i \ = \ 
\bar\sigma_{i+1}  \bar\sigma_i \bar\sigma_{i+1} 
\quad (i=1,\dots,n-1)$, 
\item \ $\bar\sigma_i \bar\sigma_j \ = \ \bar\sigma_j 
\bar\sigma_i \ \ \ (|i-j|\ge2)$, 
\item \ $\bar\sigma_1 \dotsm \bar\sigma_{n-2}\, \bar\sigma_{n-1}^2  
\bar\sigma_{n-2} \dotsm \bar\sigma_1 \ = \ 1$,
\item \ $(\bar\sigma_1 \bar\sigma_2 \dotsm \bar\sigma_{n-1})^n \ = \ 1$.
\end{itemize}
Observe that there is a natural epimorphism 
\begin{align} \label{Psi_n}
\Psi_n : B_n & \ \epi \ \gn \\
\sigma_i & \ \mapsto  \ \bar\sigma_i \qquad (i=1,\dots,n-1) \notag
\end{align}
which factors through $\mbn=B_n/C_n$.
We also write $\pgn$ for the {\it pure mapping class group of the $n$-times punctured sphere}
which is by definition the kernel of the epimorphism
\begin{align}
\gamma_n : \gn & \ \epi \ \sn \\
\bar\sigma_i & \ \mapsto  \ (i, i+1)\qquad (i=1,\dots,n-1)
\notag
\end{align}
fitting in the exact sequence
\begin{gather}
\label{exactGn}
\begin{CD}
1 @>>> \pgn @>>> \gn @>>> \sn @>>> 1.
\end{CD}
\end{gather}
In this paper, besides the above epimorphism $\Psi_n$ (\ref{Psi_n}), 
another shifted morphism
\begin{align}
\Phi_n : \bn & \ \to \ \gno \\
\sigma_i & \ \mapsto  \ \bar\sigma_i 
\qquad (i=1,\dots,n-1)  \notag
\end{align}
plays an important role, whose kernel is known to coincide with $\cn$
(\cite[\S 9.2-3]{fm}).
The homomorphism $\Phi_n$ induces the
following commutative diagram of groups 
\begin{equation}
\label{discreteCD}
\vcenter{
\xymatrix{
 1 \ar[r] &\mpn \ar[r] \ar[d]^{\vertsim}&\mbn \ar[r]^{\mpin}
 \ar@{^{(}->}[d] & \sn \ar[r] \ar@{^{(}->}[d]^{\iota_n}&1  \\
1 \ar[r] &\pgno \ar[r]  &\gno \ar[r]^{{{\gamma}_{n+1}}} &\sno\ar[r]   &1,
}
}
\end{equation}
where the horizontal sequences are exact;
the left-hand (resp. middle; right-hand) vertical arrow is the isomorphism 
(resp. the injection; the natural injection which 
trivially extends each permutation
of $\{1, 2, \dotso, n\}$ to that of $\{1, 2, \dotso, n+1\}$) 
induced from $\Phi_n$.


It is well known that the profinite completion functor preserves 
the (injectivity of the) kernel part of the exact sequences
(\ref{exactSn})-(\ref{exactPn}) and 
(\ref{exactGn}) respectively.
If $Z(G)$ denotes the center of a profinite group $G$, then
\begin{equation}
\label{various-centers}
\begin{cases}
&Z(\mpn) = Z(\mbn) = Z(\mpnh) = Z(\mbnh) = \{1\}, \\
&\cnh = Z(\pnh) = Z(\bnh) \ (\cong \zh).
\end{cases}
\end{equation}
hold (cf.\,e.g.,\,\cite[\S 1.2-1.3]{n}).

\begin{dfn}
\label{moduli}
\normalfont
Let $n\ge3$ be an integer.
We shall write $(\ast)$ for the commutative diagram of profinite groups
\begin{equation*}
\vcenter{
\xymatrix{
 1 \ar[r] &\mpnh \ar[r] \ar[d]^{\vertsim}&\mbnh \ar[r]^{\mpinh}\ar@{^{(}->}[d]& \sn\ar[r] \ar@{^{(}->}[d]^{\iota_n}&1  \\
1 \ar[r] &\pgnoh \ar[r]  &\gnoh \ar[r]^{{\widehat{\gamma}_{n+1}}} &\sno\ar[r]   &1
}
}
\tag{$\ast$}
\end{equation*}
%
which is obtained as the profinite completion of
(\ref{discreteCD}).
Note that the horizontal sequences are exact as remarked as above.
\end{dfn}


\begin{prop}
\label{char}
Suppose that $n \neq 4$, $n\ge3$.
Then every epimorphism $\mbnh \epi \sn$ has kernel $\mpnh$.
In particular, $\mpnh$ is a characteristic subgroup of $\mbnh$.
\end{prop}

\begin{proof}
E.Artin (\cite[Theorem 1]{a}) classified all surjective homomorphisms
$B_n\epi \frakS_n$ up to equivalence by conjugation in $\frakS_n$: 
When $n\ne 4,6$, there is a unique equivalence class and when $n=6$
there are two classes mutually equivalent by a nontrivial
outer automorphism of $\frakS_6$. This proves the assertion
for discrete braid groups. Lemma \ref{ker} below
with the residual finiteness of $\mbn$ settles 
the assertion for the profinite braid groups.
\end{proof}

\begin{lem}
\label{ker}
Let $G$ be a residually finite group, $N$ a normal subgroup of $G$ 
with finite quotient $Q \defeq G/N$.
Suppose that every epimorphism $G \epi Q$ has the same kernel $N$.
Then, every epimorphism $\widehat{G} \epi Q$ has the same 
kernel $\widehat{N}$.
\end{lem}


\begin{proof}
Note first that, by one-to-one correspondence between 
the finite index subgroups of $G$ and 
the open subgroups of $\widehat{G}$, 
the image of the monomorphism $\widehat{N} \to \widehat{G}$ 
coincides with the closure of $N$ in $\widehat{G}$.
Let $p : \widehat{G} \epi Q$ be a given epimorphism.
Then, 
by \cite[Proposition 3.2.2 (a)]{rz}, 
the closure of $H \defeq \Ker(p) \cap G$ in $\widehat{G}$ 
coincides with $\Ker(p)$.
Consider the composite:
\[
\varphi : G \ \epi \ G/H \ \isom \ \widehat{G}/\Ker(p) \ \isom \ Q,
\]
where the first arrow is the projection,
the second arrow is the isomorphism induced from the associated morphism $G \hra \widehat{G}$ (\cite[Proposition 3.2.2 (d)]{rz}) and
the third arrow is the isomorphism induced by $p$.
{}From the assumption, $\varphi$ has the kernel $N$, i.e.,
$N=\Ker(\varphi) = H$.
Thus, $\Ker(p)$ coincides with $\widehat{N}$.
\end{proof}

\section{Special case $\mbfh$}

The main aim of this section is to provide a proof of the following

\begin{prop} \label{charaP4B4}
 $\mpfh$ is a characteristic subgroup of $\mbfh$.
\end{prop}

\noindent
In the proof of \cite[Theorem 11]{dg} claiming 
that $\cP_n$ is characteristic in $\cB_n$ for $n\ge 3$, 
we find an inaccurate argument for the case $n=4$:
By E.\,Artin's classic work (\cite[Theorem 1]{a}), each 
surjective homomorphism
$B_4\epi \frakS_4$ is equivalent to one of the 
following $\epsilon_1,\epsilon_2,\epsilon_3$
up to change of labels in $\{1,2,3,4\}$:
\begin{align*}
\epsilon_1:B_4 & \epi \Sf \qquad
(\sigma_1  \mapsto \ (12), \
\sigma_2 \mapsto \ (23) ,\
\sigma_3 \mapsto \ (34)); \\
\epsilon_2:B_4 & \epi \Sf \qquad
(\sigma_1  \mapsto \ (1234), \
\sigma_2 \mapsto \ (2134) ,\
\sigma_3 \mapsto \ (1234)); \\
\epsilon_3:B_4 & \epi \Sf \qquad
(\sigma_1  \mapsto \ (1234), \
\sigma_2 \mapsto \ (2134) ,\
\sigma_3 \mapsto \ (4321)).
\end{align*}
Among them, $\ker(\epsilon_1)=P_4$, while 
neither $\ker(\epsilon_2)$ or $\ker(\epsilon_3)$
equals to $P_4$, for
$\sigma_1^2\in P_4$ has non-trivial images
in $\frakS_4$:
$\epsilon_2(\sigma_1^2)=
\epsilon_3(\sigma_1^2)=(13)(24)$.

Let $\bar\epsilon_1:\cB_4\epi\frakS_4$ be the 
induced map.
Given an arbitrary automorphism 
$\phi\in \aut(\cB_4)$, consider the composite
$$
\epsilon_\phi:
B_4\to B_4/C_4=\cB_4
\underset{\phi}{\isom}
\cB_4
\overset{\bar\epsilon_1}{\longrightarrow}
\frakS_4.
$$
Dyer-Grossman \cite[p.1159]{dg} discusses that
$\epsilon_\phi$ cannot be equivalent to $\epsilon_2$,
for $(\bar\sigma_1\bar\sigma_2\bar\sigma_3)^2$ 
has order exactly two in $\cB_4$ hence does not
belong to $\cP_4$ (torsion-free),
while $\epsilon_2((\sigma_1\sigma_2\sigma_3)^2)=1$.
If moreover one knew $\epsilon_\phi\not\sim \epsilon_3$,
then one could get $\epsilon_\phi\sim\epsilon_1$ 
and hence $\phi(\cP_4)=\cP_4$ so as to conclude
Proposition \ref{charaP4B4}.
However, in \cite{dg}, apparently omitted is a
discussion about $\epsilon_3$ as the existence 
of $\epsilon_3$ is 
already missed in
their citation of Artin's theorem in \cite[Theorem 2]{dg}.
Since 
$\epsilon_3((\sigma_1\sigma_2\sigma_3)^2)=(12)(34)\ne 1$,
a simple replacement of the above argument for $\epsilon_\phi\not\sim\epsilon_2$
does not work to eliminate another possibility
$\epsilon_\phi\sim\epsilon_3$.

The fact that $\mathcal{P}_4$ is a characteristic subgroup of $\mbf$
has followed in a different approach by topologists
(see, e.g., \cite[Theorem 3]{Ko}) by using
finer analysis of the mapping class group action
on the complex of curves $C(S)$ on a
topological surface $S$.
However, a profinite variant of $C(S)$ to derive Proposition 
\ref{charaP4B4} still remains
unsettled even to this day.
Below, we give an alternative argument looking closely at a family of characteristic
subgroups of $\cB_4$. We argue in the profinite context, however, 
our discussion works also  
for the discrete case in the obvious interpretation.
Our main targets arise from the following
epimorphisms  
$b_{43}:\mbfh \epi \mbth$ and $s_{43}:\Sf \epi \st$
defined by 
\begin{equation} \label{orbicurve}
b_{43}:\mbfh  \epi  \mbth :
\begin{cases}
\bar\sigma_1, \bar\sigma_3 & \mapsto \bar\sigma_1, \\
\quad \bar\sigma_2 &  \mapsto \bar\sigma_2;
\end{cases}
\quad  s_{43}:\Sf \epi \St :
\begin{cases}
 (12), (34) & \mapsto  (12) ,  \\
\quad (23) & \mapsto  (23) ,
\end{cases}
\end{equation}
and the composition
\begin{equation}
\mathfrak{P} := \mpith\circ b_{43} \left(= s_{43}\circ\widehat{\pi}_4\right)  :\mbfh \epi \st
\end{equation}
where $\widehat{\pi}_n:\widehat{\mathcal{B}}_n \epi \mathfrak{S}_n$ is as in the previous section.
The kernel of $s_{43}$ is what is called the Klein four group 
\[
V_4 \ \defeq \ \ker(s_{43}) \ = \ \{id, (12)(34), (13)(24), (14)(23)\}
\subset\frakS_4.
\]
Denote by $p_{43}: \mpfh \to \mpth$ the restriction of $b_{43}:\mbfh\to \mbth$ and write $\pzfh:=\ker(p_{43})$.
We note that $p_{43}$ is not the same as the usual 
homomorphism obtained by forgetting 
one strand of pure 4-braids.
These maps fit into the following commutative diagram of 
horizontal and vertical exact sequences:
\begin{equation}
\label{BigDiagram}
\vcenter{
\xymatrix{
{}  & 1\ar[d] &  1\ar[d] & 1\ar[d] & {} \\
1 \ar[r] & \pzfh \ar[d] \ar[r] & \ker(b_{43}) \ar[d] \ar[r] & V_4 \ar[d] \ar[r] & 1 \\
1 \ar[r]  & \mpfh \ar[r] \ar[d] ^{p_{43}}&\mbfh \ar[rd]_{\mathfrak{P}} \ar[r]^{\mpifh} \ar[d] _{b_{43}}
& \Sf \ar[r] \ar[d]^{s_{43}} &1 \\
1 \ar[r] &\mpth \ar[r] \ar[d] &\mbth \ar[r]_{\mpith} \ar[d] &\St \ar[d] \ar[r] &1 \\
{} & 1 & 1 & 1 & {}. 
}}
\end{equation}

Concerning the two sequences 
of subgroups
$\mbfh
\supset\ker(b_{43})\supset\pzfh$
and $\mbfh\supset\ker(\mathfrak{P})\supset\mpfh
$, we shall prove

\begin{prop} \label{charaKer}
(i) $\pzfh$ is a characteristic subgroup of $\ker(b_{43})$. \\
(ii) $\ker(\mathfrak{P})$ is a characteristic subgroup of $\mbfh$. \\
(iii)  $\ker(b_{43})$ is a characteristic subgroup of $\mbfh$.\\
(iv) $\pzfh$ is a characteristic subgroup of $\mbfh$.\\
(v)  $\mpfh$ is a characteristic subgroup of $\mbfh$.
\end{prop}

Proposition \ref{charaP4B4} is obtained as (v) of the above Proposition.
Here is a simple immediate consequence of it:

\begin{cor}
$\widehat{P}_n$ is a charactersitic subgroup
of $\widehat{B}_n$ for every $n\ge 3$.
\end{cor}

\begin{proof}
Proposition \ref{char} and Proposition \ref{charaP4B4} show
that $\mpnh$ is a characteristic subgroup of $\mbnh$ for every
$n\ge 3$. Assertion follows from this and 
the fact that $\widehat{P}_n$ is the inverse image 
of $\mpnh$ by the projection $\widehat{B}_n\epi\mbnh$ 
whose kernel is the center $\cnh$ of $\widehat{B}_n$.
\end{proof}

For the proof of Proposition \ref{charaKer}, note first that
(iv) follows from (i) and (iii).
We will apply (iv) for the proof of (v).
Assertion (ii) will be used to prove (iii).
In fact, (ii) follows from a stronger assertion
that every epimorphism $\mbfh \epi \st$ has 
the same kernel as $\ker(\mathfrak{P})$. In fact, it is not difficult
to see that every (discrete group) homomorphism 
$B_4\epi \st$ is conjugate to the standard one $B_4\epi B_3\epi \frakS_3$ (cf.\,e.g., 
\cite[Theorem 3.19 (a)]{lin}).
Since $\mbf$ is residually finite, the profinite version 
follows from Lemma \ref{ker}.
To complete the proof of Proposition \ref{charaKer},
it remains to prove (i), (iii) and (v).

\medskip
\noindent
 {\it Proof of Proposition} \ref{charaKer} (i): 
Let us begin with geometric interpretation of 
$\pzfh\subset \ker(b_{43})$ which has been
well studied by topologists
(see, e.g., \cite[\S 2.1]{aswy}, \cite[\S 3]{KS}).
%
One may regard the standard lift $\beta_{43}:\widehat B_4\epi \widehat B_3$ of $b_{43}:\widehat\cB_4\to\widehat\cB_3$ 
(given by $\sigma_1,\sigma_2,\sigma_3\mapsto \sigma_1,\sigma_2,\sigma_1$ respectively) 
as the $\pi_1^\et$-transform of 
the ``Cardano-Ferrari mapping 
$\mathfrak{F}_0: (\mathbb{A}^4\setminus D)_0
\to (\mathbb{A}^3\setminus D)_0$'' 
assigning to a monic quartic (with no multiple zeros)
its cubic resolvent
(in the notations of \cite[\S 5.4]{N13}). The kernel of $\beta_{43}$
is isomorphic to the free profinite group $\widehat F_2$
of rank 2. In fact, after Mordell transformation, the homomorphism 
$\beta_{43}=\pi_1^\et(\mathfrak{F}_0)$ turns to interpret
the monodromy of the universal
family of the (affine part of) elliptic curves
\begin{equation*}
\xymatrix{
{} & \mathcal{E}\setminus\{O\}= \{Y^2=4X^3-g_2X-g_3\} \ar[d] \\ 
{} & \mathfrak{M}_{1,1}^{\,\omega} =\{(g_2,g_3)\mid \Delta\ne 0\}.
}
\end{equation*}
Let $\sqrt{\zeta_4}:=(\sigma_1\sigma_2\sigma_3)^2$
so that $\beta_{43}(\sqrt{\zeta_4})=\zeta_3\in B_3$.
Then, the reduced sequence
\begin{equation}
\begin{CD}
1 @>>> F_2 @>>> B_4/\la\sqrt{\zeta_4}\ra @>>> B_3/\la\zeta_3\ra=\mathrm{PSL}_2(\Z) @>>> 1
\end{CD}
\end{equation}
fits in the orbifold quotient of the
complex model of elliptic curve family 
over the upper half plane.
Taking into account that $\sqrt{\zeta_4}\pmod{\la\zeta_4\ra}$ 
acts on each elliptic curve 
$E:Y^2=4X^3-g_2X-g_3$
by the switching $\pm Y$ involution, we see that
$\ker(b_{43})$ can be regarded as the
fundamental group of an orbicurve 
$\mathbb{P}^1_{\infty,2,2,2}$ obtained as the $X$-line 
from $(E\setminus\{O\})/\{\pm 1\}$;
it turns out to be isomorphic to
the profinite free product of three copies of $\Z/2\Z$:
\begin{align}
\ker(b_{43})&=\pi_1(\mathbb{P}^1_{\infty,2,2,2}/\C) \\
&=(\Z/2\Z)\freeprod(\Z/2\Z)\freeprod(\Z/2\Z) 
\notag
\end{align}
which may also be regarded as the profinite completion 
of discrete free product 
$(\Z/2\Z)\ast(\Z/2\Z)\ast(\Z/2\Z)$
(\cite[\S 9.1]{rz}).
The normal subgroup $\pzfh$ of $\ker(b_{43})$ corresponds
to the fundamental group of the 
Galois cover of $\mathbb{P}^1_{\infty,2,2,2}$
with group $V_4$ given in 
the Latt\'es cover diagram:
\begin{equation}
\vcenter{
\xymatrix{
E\setminus E[2] \ar[d] \ar[r] & \mathbb{P}^1
-\{e_0,e_1,e_2,e_3\} \ar[d] \\ 
E\setminus \{O\} \ar[r] &\mathbb{P}^1_{\infty,2,2,2}
}}
\end{equation}
where the left vertical arrow is the isogeny
of punctured elliptic curves by multiplication by 2,
and horizontal arrows correspond to 
the $\{\pm 1\}$-quotients.
{}From this we obtain a cartesian diagram of profinite groups:
\begin{equation}
\label{Pi04inkeralpha}
\vcenter{
\xymatrix{
\ker(b_{43})=(\Z/2\Z)\freeprod(\Z/2\Z)\freeprod(\Z/2\Z) 
\ar @{->>}[r]
& (\Z/2\Z)\times(\Z/2\Z)\times(\Z/2\Z)  \\
\widehat\Pi_{0,4} \ar @{_{(}->}[u]\ar @{->>}[r]
&(\Z/2\Z) \ar @{_{(}->}[u]_{{}_{diagonal \ map}},
}
}
\end{equation}
where the upper horizontal arrow is the abelianization map.
%
Moreover, according to Herfort-Ribes (\cite[Theorem 2 (i)]{HR}),
the torsion elements of  
$(\Z/2\Z)\freeprod(\Z/2\Z)\freeprod (\Z/2\Z)$
form exactly the three conjugacy classes of order two
which, therefore, must be preserved as a set 
under $\aut(\ker(b_{43}))$.
This characterizes the diagonal image of $(\Z/2\Z)$
in the right hand side of (\ref{Pi04inkeralpha}).
Thus we conclude that
$\widehat\Pi_{0,4}$ is characteristic 
in $\ker(b_{43})$ as the pull-back image of 
$(\Z/2\Z)\overset{diag.}{\hookrightarrow} (\Z/2\Z)^3$
along the abelianization of $\ker(b_{43})$.
\qed

\medskip
\noindent
 {\it Proof of Proposition} \ref{charaKer} (iii):
To prove (iii), pick any $\phi \in \aut(\mbfh)$.
We first show that $\phi(\ker(b_{43})) \subset \ker(b_{43})$.
As $\ker(\mathfrak{P})$ is characteristic in $\widehat\cB_4$
as shown in (ii),  it follows that
$\phi(\ker(b_{43}))\subset \ker(\mathfrak{P})$.
Hence $b_{43}$ maps $\phi(\ker(b_{43}))$ onto 
a subgroup of $\mpth(\isom\pgtoh\cong\widehat F_2$).
But $\phi(\ker(b_{43}))$ is isomorphic to
$\ker(b_{43})$ which is a topologically finitely 
generated closed normal subgroup of $\mbfh$.
Since $\widehat F_2$ has no nontrivial non-free 
finitely generated normal subgroups 
(\cite[Corollary 3.14]{LvD})
and 
since $\ker(b_{43})\cong (\Z/2\Z)
\freeprod
(\Z/2\Z)\freeprod(\Z/2\Z)$ has finite 
abelianization $(\Z/2\Z)^3$, the image $\phi(\ker(b_{43}))$ must be annihilated
by $b_{43}$, i.e., $ \phi(\ker(b_{43}))\subset \ker(b_{43})$.
We can argue in the same way after replacing $\phi$ 
by $\phi^{-1}$ to obtain 
$ \phi^{-1}(\ker(b_{43}))\subset \ker(b_{43})$.
Combining both inclusions implies
$ \phi(\ker(b_{43}))=\ker(b_{43})$.
\qed

\medskip
\noindent
 {\it Proof of Proposition} \ref{charaKer} (v):
Let us write $[\ast]^\ab$ for the abelianization of $[\ast]$.
Since we already know `(iv): $\pzfh$ is characteristic in $\mbfh$'
from (i)-(iii), for proving $\mpfh$ characteristic in $\mbfh$,
it suffices to show the assertion 
that $ \mpfh$ is the kernel
of the conjugate representation
$\rho:\mbfh \ {\rightarrow} \ \aut(\pzfh^{\mathrm{ab}})$.
First we note that $\rho$ factors through 
$\bar\rho: \mbfh/\mpfh\cong\frakS_4\to
\aut(\pzfh^{\mathrm{ab}})$. 
This follows from the observation
that $p_{43}^\ab$ injects
$\pzfh^{\mathrm{ab}}$ into $\mpfh^\ab$: 
Indeed, writing $\{\bar x_{ij}\}$ for the image of
the standard generator system
$\{x_{ij}=\sigma_{j-1}\cdots\sigma_{i+1}\sigma_i^2
\sigma_{i+1}^{-1}\cdots\sigma_{j-1}^{-1}\mid
1\le i<j\le n\}$ of 
$\cP_n$, we find
\begin{equation} \label{p34ab}
p_{43}^\ab:\mpfh^\ab \to \mpth^\ab :\ 
\begin{cases}
\bar x_{12}, \bar x_{34} & \mapsto \bar x_{12}, \\
\bar x_{13}, \bar x_{24} &  \mapsto \bar x_{13}, \\
\bar x_{14}, \bar x_{23} &  \mapsto \bar x_{23}.
\end{cases}
\end{equation}
Taking into account the single relation
$\bar x_{12}+ \bar x_{13} +\bar x_{14}
+\bar x_{23}+\bar x_{24}+\bar x_{34}=0$ for 
$\mpfh^\ab$
(respectively, $\bar x_{12}+ \bar x_{13} +\bar x_{14}=0$ for 
$\mpth^\ab$), we easily see from the description
(\ref{p34ab}) of $p_{43}^\ab:\zh^5\epi \zh^2$
 that $\ker(p_{43}^\ab)$
is isomorphic to $\zh^3$ (torsion-free) 
into which $\pzfh^\ab$ must inject.
Then, to complete proof of the assertion, 
it suffices to see faithfulness of  
$\bar\rho: \mbfh/\mpfh\cong\frakS_4\to
\aut(\pzfh^{\mathrm{ab}})$.
This is easily seen from the general fact that the action of 
$\cB_n/\cP_n=\frakS_n$ on
the $\bar x_{ij}\in\mathcal{P}^\ab_n$ is 
given by the natural action on indices, once 
declared $\bar x_{ij}=\bar x_{ji}$.
The action of $\frak S_4$ on $\pzfh^\ab$ turns
out to be the standard permutation representation
$\zh^4$ modulo the diagonal line, which is
faithful. \qed

\section{Proofs of Theorems A and B}

By virtue of  Propositions \ref{char} and \ref{charaP4B4}, we know that
$\mpnh $ is a characteristic subgroup of $\mbnh$ for
$n\ge 3$.
The following proposition follows immediately from this together with
the well-known fact that $\out(\sn) = \{1\}$ in the case $n\ne 6$.
However, the case $n=6$
requires a special care, since 
$\out(\mathfrak{S}_6) \cong \zz/2\zz$.
Theorem \ref{hmm} (Hoshi-Minamide-Mochizuki)
allows us to give a uniform proof working 
for all $n\ge 4$.


\begin{prop}
\label{out} 
Regard $\sn$ as the quotient of 
$\mbnh$ and of $\widehat\Gamma_{0,[n]}$
by $\varpi_n:B_n\to\sn$ in \S 2. \\
(i) Every automorphism of $\mbnh$ induces an inner automorphism of $\sn$ for $n\ge 3$. \\
(ii) $\widehat\Gamma_{0,n}$ is a characteristic subgroup of $\widehat\Gamma_{0,[n]}$ in the profinite completion of (\ref{exactGn}),
and every automorphism of $\widehat\Gamma_{0,[n]}$ induces an inner automorphism of $\sn$ for $n\ge 5$.
\end{prop}

\begin{proof}
(i) As $\out(\St) = \{1\}$, the assertion is trivial when
$n=3$. Suppose $n\ge4$ and pick any $\phi \in \aut(\mbnh)$.
Then, it follows from 
Propositions \ref{char} and \ref{charaP4B4}, 
that $\phi$ induces 
$(\phi_P,\phi_\frakS)
\in\aut(\mpnh)\times\aut(\frakS_n)$,
Moreover $\phi_P$ induces $\phig\in\aut(\pgnoh)$ 
via the natural isomorphism $\mpnh \isom \pgnoh$
given by $\Phi_n$ of \S 2. 
Let $\bar\phi_\Gamma\in\out(\pgnoh)$ be the outer
class of $\phi_\Gamma$, and 
let $(\phi_0,\phi_1)\in\gt\times\frakS_{n+1}$ be the image
of $\bar\phi_\Gamma$ under the 
isomorphism $\out(\pgnoh)\isom\gt\times\frakS_{n+1}$ of
Theorem \ref{hmm}. Then we have the commutative diagram
\begin{equation}
\label{diag41}
\vcenter{
\xymatrix{
\sn \ar[r]^-{\chi_n} \ar[d]_{\vertsim}^{\phis}
& \out(\pgnoh) \ar[r]^{\sim} \ar[d]^{\inn(\bar{\phi}_{\Gamma})}_{\vertsim}
& \gt\times\frakS_{n+1} 
\ar[d]^{\inn(\phi_0,\phi_1)
}_{\vertsim}
\\
\sn \ar[r]^-{\chi_n} &\out(\pgnoh) \ar[r]^{\sim} 
& \gt\times\frakS_{n+1}
}}
\end{equation}
where $\chi_n : \sn\to \out(\mpnh)=\out(\pgnoh)$
is the natural isomorphism
regarding the commutative diagram
$(*)$ in Definition \ref{moduli}.
Since $\chi_n$ factors through
$\iota_n : \sn \hra \sno$, the above (\ref{diag41}) makes the diagram
\[
\begin{CD}
\sn @>\iota_n>> \sno \\
@V{\phis}V{\vertsim}V 
@V{\vertsim}V\inn(\phi_1)V \\
\sn @>\iota_n>> \sno
\end{CD}
\]
commutative, hence $\phi_1
$ normalizes (hence lies in)
the image of $\iota_n$.  
{}From this follows that $\phis$ is an inner automorphism of $\sn$.

(ii): Recall from \S 2 that there is a surjection sequence
$\bnh \epi \mbnh \epi \widehat\Gamma_{0,[n]}\epi \frakS_n$.
By Proposition \ref{char},
every epimorphism from $\mbnh$ to $\frakS_n$ has 
kernel $\mpnh$ for $n\ge 5$.
This makes $\widehat\Gamma_{0,n}$ to be a 
characteristic subgroup of $\widehat\Gamma_{0,[n]}$ 
as the pull-back of $\mpnh \subset \mbnh$. 
For the rest, we can argue in exactly a similar (and simpler) way to
the case (i) with employing
$\chi_n':\sn\to\out(\widehat\Gamma_{0,n})\cong\gt\times\frakS_n$
for the role of $\chi_n$ in (i).
We leave the rest of detail to the reader.
\end{proof}

For the proof of Theorem B, we prepare a simple 
lemma of group theory.
Let 
\[
\begin{CD}
1 @>>> \Delta @>>> \Pi @>>> G @>>> 1
\end{CD}
\]
be an exact sequence of profinite groups with 
$\rho : G \to \out(\Delta)$ 
the associated outer representation.
Let $Z_{\out(\Delta)}(\im(\rho))$ denote the centralizer
of the image $\rho(G)$ in $\out(\Delta)$.
Assume that $\Delta$ and $G$ are topologically finitely generated
so that $\aut(\Delta)$, $\aut(G)$ are profinite groups.
Write $\aut_G(\Pi)$ (resp. $\inn_\Pi(\Delta)$) 
for the group of automorphisms of $\Pi$ 
which preserve $\Delta \sseq \Pi$ and
induce the identity automorphism of $G$
(resp. for the group of inner automorphisms 
of $\Pi$ by the elements of $\Delta$).
Then,
\begin{lem}
\label{dpg}
Notations being as above, we have the following assertions.\\
(i) Suppose $Z(\Delta)=\{1\}$.
Then the restriction map $\aut_G(\Pi) \to \aut(\Delta)$ induces an isomorphism
\[
\aut_G(\Pi)/\inn_\Pi(\Delta) \ \isom \ Z_{\out(\Delta)}(\im(\rho)).
\]
(ii)
Suppose $Z(G)=\{1\}$ and that $\Delta$ is a characteristic subgroup of $\Pi$.
Then we have an exact sequence of profinite groups
\[
\begin{CD}
1 @>>> \aut_G(\Pi)/\inn_\Pi(\Delta) @>\mathj>> \out(\Pi) @>\varpi>> \out(G).
\end{CD}
\] 
\end{lem}

\begin{proof}
Assertion (i) follows immediately from \cite[Corollary 1.5.7]{n}.
We consider (ii).
First, observing
$
\aut_G(\Pi) \ \cap \ \inn(\Pi) \ = \ \inn_\Pi(\Delta)
$
under the assumption $Z(G)=\{1\}$,
we obtain the monomorphism
\[
\mathj : \aut_G(\Pi)/\inn_\Pi(\Delta)
\hra \ \aut(\Pi)/\inn(\Pi) \ = \ \out(\Pi)
\]
from the natural injection $\aut_G(\Pi) \hra \aut(\Pi)$.
Next, since $\Delta$ is a characteristic subgroup of $\Pi$, 
there exists a natural homomorphism
$\varpi : \out(\Pi) \ \to \ \out(G)$
with $\varpi\circ\mathj=1$.
Then, immediately from the surjectivity
$\inn(\Pi)\epi \inn(G)$ follows
that $\im(\mathj) = \Ker(\varpi)$, which 
completes the proof of (ii).
\end{proof}

We now obtain Theorem B:

\begin{thm}
\label{gtmbnh} 
(i)
Let $n\ge4$ be an integer.
Then the composite
\[
\gt \ \to \ \out(\mbnh)
\]
of the natural homomorphisms $\gt \to \out(\bnh)\to \out(\mbnh)$ is an isomorphism.
\\
(ii) Let $n\ge 5$. Then, the natural homomorphism
\[
\gt \ \to \ \out(\widehat\Gamma_{0,[n]})
\]
induced from $\widehat{\Psi}_n : \widehat{B}_n  \epi \widehat\Gamma_{0,[n]}$ (\ref{Psi_n}) is an isomorphism.
%
\end{thm}

\begin{proof}
First, we note that $\sn$ and $\mpnh$ are center-free 
(\ref{various-centers}), 
and that $\mpnh$ is a characteristic subgroup of $\mbnh$ 
(Propositions \ref{char} and \ref{charaP4B4}).   
Consider the upper exact sequence
\[
\begin{CD}
1 @>>> \mpnh @>>> \mbnh @>>> \sn @>>> 1 \\
\end{CD}
\]
of $(*)$ in Definition \ref{moduli}, and write
$\varphi_n:\sn \to \out(\mpnh)$ for
the associated outer representation.
Let us apply Lemma \ref{dpg} to the above exact sequence.
By virtue of Proposition \ref{out} (i), the homomorphism
$\varpi: \out(\mbnh)\to \out(\sn)$ of Lemma \ref{dpg} (ii)
turns out trivial, so $\mathj$ in loc.\,cit. together with
Lemma \ref{dpg} (i) gives an isomorphism
\[
Z_{\out(\mpnh)}(\varphi_n(\sn)) \ \isom \ \out(\mbnh).
\]
Then observe that the natural isomorphism $\mpnh \isom \pgnoh$ 
in $(*)$
induces an isomorphism
\[
Z_{\out(\mpnh)}(\varphi_n(\sn)) \ \isom \ Z_{\out(\pgnoh)}(\chi_n(\sn)),
\]
where $\chi_n: \sn \to \out(\pgnoh)$ is as in (\ref{diag41}).
But since $\iota_n(\sn)$ has trivial centralizer in
$\frakS_{n+1}$, Theorem \ref{hmm} implies
\[
\gt \ \isom \ Z_{\out(\pgnoh)}(\chi_n(\sn)).
\]
It is easy to see that the composite of the above three displayed isomorphisms
coincides with $\gt \to \out(\mbnh)$ of the assertion.
This completes the proof of (i).

(ii) Let $n\ge 5$. 
After Proposition \ref{out} (ii), the argument goes in a 
similar (and simpler) way to the case (i) 
with applying Lemma \ref{dpg} to the profinite completion of (\ref{exactGn}):
$$
\begin{CD}
1 @>>> \widehat\Gamma_{0,n} @>>> \widehat\Gamma_{0,[n]} @>>> \sn @>>> 1.
\end{CD}
$$
We leave the rest of detail to the reader.
\end{proof}

Now, to prove Theorem A, let us follow an argument in \cite{dg} (Theorem 20)
to look closely at the short exact sequence
\begin{equation}\label{Wells0}
\begin{CD}
1 @>>> \cnh @>>> \bnh @>>> \mbnh @>>> 1
\end{CD}
\end{equation}
obtained as the profinite completion of 
(\ref{exactBn}).
Since $\cnh$ is characteristic in $\bnh$, 
this yields two natural homomorphisms 
\begin{equation}\label{WellsPair}
\fp_0:\aut(\bnh)\to\aut(\cnh),
\quad
\fp_1:\aut(\bnh)\to\aut(\mbnh).
\end{equation}
Recalling $\cnh=\la \zeta_n\ra\cong\zh$, 
we now {\it canonically} identify $\aut(\cnh)=\zh^\times$.

\begin{dfn}
For $n> 1$, define the subgroup $Z_n\subset \zhat^\times$ by
$$
Z_n :=
\bigl(1+n(n-1)\zhat\bigr)^\times=
\ker\left(\zhat^{\times} \epi (\zhat/n(n-1)\zhat)^{\times}\right).
$$
\end{dfn}
\noindent
It is clear that each $\nu \in Z_n$ has a unique element 
$e \in \zhat$ such that
\[
\nu=  1 + n(n-1)e.
\]
(But note that this form of $\nu$ is not always in $\zh^{\times}$ for arbitrary $e \in \zh$.)

The next key lemma enables us to identify $\ker(\fp_1)$ with $Z_n$:

\begin{lem} \label{lemma_phi_nu}
There is an isomorphism 
$$
\phi: Z_n \isom \ker(\fp_1)\subset \aut(\bnh)
$$ 
which assigns to every $\nu\in Z_n$ an automorphism $\phi_\nu\in \aut(\bnh)$
determined by 
$$
\phi_\nu(\sigma_i)=\sigma_i\zeta_n^e 
\qquad (\nu=  1 + n(n-1)e, \ i=1,\dots,n-1).
$$
\end{lem}

\begin{proof}
Given any $\nu \in Z_n$, let 
$e \in \zhat$ be the unique element with $\nu = 1 + n(n-1)e$. 
By using this $e \in \zh$, we define $\phi_{\nu} \in \Ker(\mathfrak{p}_1)$ as follows: First,
set $\phi_\nu(\sigma_i):=\sigma_i\zeta_n^e$ for  all $i=1,\dots,n-1$.
Since $\zeta_n$ lies in the center of $\bnh$, it is easy to see that
$\phi_\nu$ preserves the Artin's braid relations. Therefore, $\phi_\nu$
extends to an endomorphism of $\bnh$ written by the same symbol
$\phi_\nu$. One computes then
\begin{equation} \label{phinu_zeta_n}
\phi_\nu(\zeta_n)=
\phi_\nu((\sigma_1\cdots\sigma_{n-1})^n)
=(\sigma_1\cdots\sigma_{n-1})^n\cdot \zeta_n^{n(n-1)e}
=\zeta_n^{1+n(n-1)e}=\zeta_n^\nu.
\end{equation}
{}From this we see that the image of $\phi_\nu$ contains $\zeta_n=(\zeta_n^\nu)^{\nu^{-1}}$
and hence does contain $\sigma_i=\phi_\nu(\sigma_i)\cdot\zeta_n^{-e}$
for $i=1,\dots,n-1$. This means the endomorphism $\phi_\nu:\bnh \to \bnh$
is a surjective homomorphism. As $\bnh$ is hopfian, 
we conclude $\phi_{\nu}\in \aut(\bnh)$ (cf.\,\cite[Proposition 2.5.2]{rz}).
Write $\phi: Z_n \to \aut(\bnh)$ for the homomorphism defined by the 
above correspondence 
$\nu  \mapsto \ \phi_{\nu}$.
One verifies immediately that $\phi$ is injective 
and $Z_n \cong \mathrm{Im}(\phi) \subset \Ker(\mathfrak{p}_1)$.
To see $\mathrm{Im}(\phi) = \Ker(\mathfrak{p}_1)$, pick any $\alpha\in\Ker(\fp_1)$ and 
set $\nu:=\fp_0(\alpha)\in\zh^\times$.
%
Then, $\alpha(\zeta_n)=\zeta_n^\nu$ and there exist $e_i\in\zh$ ($i=1,\dots,n-1$)
such that $\alpha(\sigma_i)=\sigma_i\, \zeta_n^{e_i}$.
It is easy to see from the braid relation that all $e_i$ are the same
constant $e \in \zh$.
But then, since $\zeta_n=(\sigma_1\cdots\sigma_{n-1})^n$, we find
$\nu=1+n(n-1)e$ which belongs to $Z_n$ and that $\alpha = \phi_{\nu}$.
\end{proof}

Theorem A is obtained from Theorem \ref{gtmbnh} (i)
together with the last part of the following

\begin{thm} \label{ThA_detail}
Let $n\ge 4$ be an integer. 
\begin{enumerate}
\item[(i)] There exists an exact
sequence 
\begin{equation*}
\label{aut-result}
\begin{CD}
1 @>>> Z_n @>\phi>> \aut(\bnh) @>\fp_1>> \aut(\mbnh) @>>> 1.
\end{CD}
\end{equation*}

\item[(ii)] $\inn(\bnh)\cap \phi(Z_n)=\{1\}$.

\item[(iii)]
The exact sequence (i) provides a split central extension, i.e., 
$$ 
\aut(\bnh) \cong \aut(\mbnh) \times Z_n,
$$
and gives rise to $\out(\bnh) \cong \out(\mbnh) \times Z_n$.
\end{enumerate}
\end{thm}

\begin{proof}
(i) It suffices to show $\fp_1$ is surjective. Note that
$\inn(\bnh)$ is mapped onto $\inn(\mbnh)$.
On the other hand, there is a well-known action 
$\iota_n:\gt \to \aut(\bnh)$ in the form
\begin{equation}
\label{GTaction}
\begin{cases}
\sigma_1 & \mapsto \ \sigma_1^\lambda, \\
\sigma_i & \mapsto \ f(\sigma_i,\zeta_i)\sigma_i^\lambda f(\zeta_i,\sigma_i) 
\quad (i=1,\dots, n-1)
\end{cases}
\end{equation}
with $(\lambda,f)\in \zh^\times \times [\widehat F_2, \widehat F_2]$ 
the standard parameter for the elements 
of $\gt$ (\cite{Dr}, \cite{I90}, \cite{im}).
Let $\bar\iota_n: \gt\to \aut(\mbnh)$ be the induced action.
By virtue of Theorem \ref{gtmbnh} (i), $\gt\cong \out(\mbnh)$,
hence $\aut(\mbnh)=\bar\iota_n(\gt)\cdot \inn(\mbnh)$.
{}From this follows that $\fp_1$ maps 
$\iota_n(\gt)\cdot \inn(\bnh) (\subset \aut(\bnh))$
onto $\aut(\mbnh)$.

(ii) This is a consequence of Lemma \ref{dpg} (ii) 
applied to (\ref{Wells0}). 
Here is an alternative direct proof: 
Recall that the abelianization
$\bnh^\ab$ of $\bnh$ is isomorphic to $\zh$.
Each inner automorphism acts trivially on $\bnh^\ab$, 
while $\phi_\nu\in \phi(Z_n)$ ($\nu\in Z_n$) acts on it by 
$$
(\sigma_i)^\ab
\mapsto (\sigma_i\cdot \zeta_n^e)^\ab=(\sigma_i^\ab)^{1+n(n-1)e}
\qquad(i=1,\dots,n-1)
$$ 
which is nontrivial unless $e=0$.
This concludes the assertion.

(iii) It follows from (ii) that $\fp_1$ induces 
$\inn(\bnh)\isom \inn(\mbnh)$.
Since $\iota_n(\gt)\isom \bar\iota_n(\gt)$,
we find from Theorem \ref{gtmbnh} (i)
that  $\fp_1$ restricts to 
the isomorphism
\begin{equation} 
\label{compl_factor}
\iota_n(\gt)\cdot\inn(\bnh)\isom
\aut(\mbnh) ,
\end{equation}
i.e.,  $\iota_n(\gt)\cdot\inn(\bnh)$ gives 
a complementary factor of $\phi(Z_n)$ 
in $\aut(\bnh)$.
To see that the exact sequence (i) gives a central extension,
it suffices to show that both $\inn(\bnh)$ and $\iota_n(\gt)$
commutes with $\phi(Z_n)$.
The commutativity of $\inn(\bnh)$ and $\phi(Z_n)$ 
follows immediately from the definition of $\phi_\nu$
($\nu\in Z_n)$
in Lemma \ref{lemma_phi_nu}.
The commutativity of $\iota_n(\gt)$ and $\phi(Z_n)$
also follows from direct computation by using
the above description 
of the $\gt$-action on $\bnh$. 
Indeed, given $(\lambda,f)\in\gt$, 
noting that $\zeta_n$ lies in the center of $\widehat{B}_n$,
and $f$ lies in the commutator subgroup of $\widehat F_2$, we have
$f(\sigma_i,\zeta_i)=f(\sigma_i\zeta_n^e,\zeta_i)$
($i=1,\dots, n-1$). Since $(\lambda,f)\in\gt$ is known to
act on $\zeta_n$ by $\zeta_n\mapsto \zeta_n^\lambda$ under
the action (\ref{GTaction}), one computes:
\begin{align*}
(\lambda,f)\circ\phi_\nu(\sigma_i) &=(\lambda,f)(\sigma_i\zeta_n^e)
=f(\sigma_i,\zeta_i)\sigma_i^\lambda f(\zeta_i,\sigma_i)
\zeta_n^{\lambda e},\\
&=f(\sigma_i,\zeta_i)(\sigma_i \zeta^e)^\lambda f(\zeta_i,\sigma_i)
=\phi_\nu(f(\sigma_i,\zeta_n)\sigma_i^\lambda f(\zeta_i,\sigma_i)) \\
&=\phi_\nu\circ(\lambda,f)(\sigma_i).
\end{align*}
for every $i=1,\dots, n-1$ (we understand $\zeta_1=1$ when $i=1$).
Thus we settle the first assertion 
$\aut(\bnh) \cong \aut(\mbnh) \times Z_n$ 
after identifying $Z_n\cong \phi(Z_n)\subset \aut(\bnh)$
and $\aut(\mbnh)\cong \iota_n(\gt)\cdot\inn(\bnh)\subset \aut(\bnh)$
via (\ref{compl_factor}).
The second assertion is then just a consequence of it. 
\end{proof}
\medskip
In our above discussion for the proof of Theorem A, 
important roles have been played by
the pair of two maps (\ref{WellsPair}),
which was motivated from
the profinite Wells exact sequence 
(cf.\,\cite[\S 1.5]{n}, \cite{JL})
associated to the short exact sequence (\ref{Wells0})
in the form:
\begin{equation}
\label{wells}
\begin{CD}
0 \longrightarrow 
\coz^1(\mbnh, \cnh) \longrightarrow 
\aut( \bnh, \cnh) 
@>\fp>> \mathscr{C} @>q>> \co^2(\mbnh, \cnh).
\end{CD}
\end{equation}
Since $\bnh$ in (\ref{Wells0}) is a central extension and
$\cB_n^\ab\cong\Z/n(n-1)\Z$,
we easily see that $\coz^1(\mbnh, \cnh) = \{0\}$, 
$\aut(\bnh, \cnh) = \aut(\bnh)$, and find
the group of ``compatible pairs'' 
$\mathscr{C}$ to be $\aut(\mbnh) \times \aut(\cnh).$
\noindent
Thus, 
the exact sequence (\ref{wells}) is reduced to 
\begin{equation}
\label{reduced-wells}
\begin{CD}
0 @>>> \aut(\bnh) @>\fp=(\fp_1,\fp_0)>> \aut(\mbnh) \times \aut(\cnh) @>q>> \co^2(\mbnh, \cnh),
\end{CD}
\end{equation}
where $q$ is called the Wells pointed map 
(generally not a homomorphism).

The above sequence (\ref{reduced-wells}) is simply useful, for example, to see that 
{\it the exact sequence of} Theorem \ref{ThA_detail} (i) 
{\it provides a central extension},
reproving the core part of Theorem \ref{ThA_detail} (iii) 
without use of the explicit $\gt$-action (\ref{GTaction}):
Indeed, according to (\ref{phinu_zeta_n}),
the image $\fp(\phi_\nu)=(\fp_1(\phi_\nu),\fp_0(\phi_\nu))=(id, \nu)$
for every $\nu\in Z_n$ is easily seen to lie
in the center of $\aut(\mbnh) \times \aut(\cnh)$.
Besides this simple observation, it is a natural question to measure 
the size of the image of  $\aut(\bnh)$ by the injection $\fp=(\fp_1,\fp_0)$
into $\aut(\mbnh) \times \aut(\cnh)$.
Now, recalling 
$\gt\subset\{(\lambda,f)\in\zhat^\times \times \widehat F_2\}$,
$\aut(\mbnh) = \gt \cdot \inn(\cB_n)$ and
$\aut(\cnh) = \zh^\times$, 
we define two characters
\begin{equation}
\lambda:\aut(\mbnh)\to \zh^\times
\quad\text{and}\quad
\nu:\aut(\cnh)\to \zh^\times
\end{equation}
in the obvious way.
One finds:
\begin{prop} Notations being as above, we have
$$
\im(\fp)=\{(\alpha,\beta)\in \mathscr{C}\mid
\lambda(\alpha)\equiv \nu(\beta) \mod n(n-1)
\}.
$$
In particular, $\mathscr{C}/\im(\fp)\cong (\Z/n(n-1)\Z)^\times$.
\end{prop}
\begin{proof}
Let $Z_n\subset\aut(\cnh)=\zh^\times$ be as above,
and define $A_n\subset\aut(\cB_n)$ to be $\lambda^{-1}(Z_n)$.
It is not difficult to see $A_n\times Z_n\subset\im(\fp)$.
The assertion is derived from the observation 
that the image of $\im(\fp)$ in the quotient
group $\mathscr{C}/(A_n\times Z_n)\cong
 (\Z/n(n-1)\Z)^\times\times(\Z/n(n-1)\Z)^\times$
forms the diagonal subgroup.
This follows from the well-known fact that the restriction
of the action of $(\lambda,f)\in\gt$ on $\bnh$
to $\cnh = \la\zeta_n\ra$ is given by
$\zeta_n\mapsto \zeta_n^\lambda$, which completes
the proof.
\end{proof}

Before closing the paper, let us add some remark on the Wells map
$q:\mathscr{C} \to \co^2(\mbnh, \cnh)$.
Let $[\mu] \in \co^2(\mbnh, \cnh)$ be
the class of factor sets associated to the central
extension (\ref{Wells0}). 
For each pair $(\alpha,\nu)\in\mathscr{C}=
\aut(\mbnh) \times \aut(\cnh)$, we denote 
by $[\mu]^{(\alpha,\nu)}\in \co^2(\mbnh, \cnh)$
the class of a central extension obtained by
twisting (\ref{Wells0}) by $(\alpha, \nu)$.
Then, one finds: 
\begin{equation}
q(\alpha,\nu)=[\mu]-[\mu]^{(\alpha,\nu)}.
\end{equation}
This means that $\im(\fp)\subset\mathscr{C}$
can be characterized as the stabilizer of the
twisting action of $\mathscr{C}$ on $[\mu]$.
Concerning the precise position and size of 
$[\mu] \in \co^2(\mbnh, \cnh)$, we remark
the following

\begin{prop}
Let $n\ge 4$. 
The cohomology group $\co^2(\mbnh, \cnh) $
is isomorphic to $\Z/n(n-1)\Z$, and is generated by
the class $[\mu]$.
\end{prop}

\begin{proof}
According to V.Arnold \cite{A68}, $H^2(B_n,\Z)=\{0\}$ and 
$H^3(B_n,\Z)=\Z/2\Z$. 
Applying this to the long exact sequence 
associated with $0\to\Z\to\Z\to\Z/r\Z\to 0$
($r\in \N$),
we obtain $H^2(B_n, \Z/r\Z)\cong \{0\}$, 
$\cong\Z/2\Z$ according to whether 
$r$ is odd or even respectively.
For a positive integer $N$, (part of) the five term exact sequence
for the central extension $1\to C_n\to B_n\to \mathcal{B}_n\to 1$
reads
\begin{align} \label{5term0}
&H^1(B_n,\Z/N\Z)
\overset{\mathrm{res}_N}{\longrightarrow} 
H^1(C_n,\Z/N\Z)
\overset{\mathrm{tg}_N}{\longrightarrow} 
H^2(\mathcal{B}_n,\Z/N\Z) 
\\
&\overset{\mathrm{inf}_N}{\longrightarrow} 
H^2(B_n, \Z/N\Z),
\notag
\end{align}
where $\mathrm{res}_N$, $\mathrm{tg}_N$ and $\mathrm{inf}_N$
are respectively the restriction, transgression and inflation
maps.
Suppose first $N$ is a positive integer divisible by $n(n-1)$.
Then, (\ref{5term0}) yields the exact sequence
\begin{equation} \label{5term}
\begin{CD}
0 \to 
\Z/n(n-1)\Z 
@>{{\overline{\mathrm{tg}}_N}}>> 
H^2(\mathcal{B}_n,\Z/N\Z) 
@>{\mathrm{inf}_N}>>
H^2(B_n, \Z/N\Z) \ (\cong \Z/2\Z),
\end{CD}
\end{equation}
where $\Z/n(n-1)\Z$ is regarded as
the cokernel of the restriction
$\mathrm{res}_N: \mathrm{Hom}(B_n,\Z/N\Z)\to \mathrm{Hom}(C_n,\Z/N\Z)$
followed by the factorization ${\overline{\mathrm{tg}}_N}$
of transgression ${\mathrm{tg}}_N$.
Let us vary $N$ multiplicatively.
The goodness of $\mathcal{B}_n$ (in the sense of Serre) 
together with \cite[Corollary 2.7.6]{NSW} allows us to interpret
$\co^2(\mbnh, \cnh)
=\varprojlim_N H^2(\mathcal{B}_n,\Z/N\Z)$
after identification $C_n=\zeta_n^\Z\cong \Z$ with trivial
(conjugate) action of $\mathcal{B}_n$.
The term $\mathrm{coker}(\mathrm{res}_N)\cong \Z/n(n-1)\Z$ 
in (\ref{5term}) is constant in the projective system 
along $N\in \N$ divisible by $n(n-1)$.  
On the other hand, we have (\#):
$\varprojlim_N H^2(B_n,\Z/N\Z)=\{0\}.$
In fact,
since $B_n^\ab\cong\Z$,
in the long exact sequence
associated with $0\to \Z/2\Z\to \Z/2N\Z\to \Z/N\Z\to 0$,
we find that 
$H^1(B_n,\Z/2N\Z)\to H^1(B_n,\Z/N\Z)$ is surjective,
hence that the former arrow in
$H^2(B_n,\Z/2\Z)\to H^2(B_n,\Z/2N\Z)\to H^2(B_n,\Z/N\Z)$ 
gives an isomorphism between groups of order two
so that the latter arrow is 0-map.
This settles (\#) which concludes the first assertion
$\co^2(\mbnh, \cnh) \cong\Z/n(n-1)\Z$.

It remains to show that the class $[\mu]$ has order $n(n-1)$ in
$\co^2(\mbnh, \cnh) $.
For an integer $d>0$, let $[\mu_d]\in H^2(\mathcal{B}_n,\Z/d\Z)$
be the class of factor sets corresponding to the central extension
\begin{equation} \label{extension_d}
1\to C_n/C_n^d(\cong \Z/d\Z)\to B_n/C_n^d\to \mathcal{B}_n\to 1.
\end{equation}
It is known that $[\mu_d]$ is the transgression image of the
projection $\mathrm{pr}_d:C_n\to C_n/C_n^d$ regarded as an element of
$H^1(C_n,C_n/C_n^d)$, i.e., 
\begin{equation} \label{suzuki-id}
[\mu_d]=\mathrm{tg}_d(\mathrm{pr}_d)\in
\im(\mathrm{tg}_d)\subset  H^2(\mathcal{B}_n,\Z/d\Z),
\end{equation}
where $C_n/C_n^d\isom \Z/d\Z$ is given by
$\zeta_n\mapsto 1$
(cf.\,e.g.,\,\cite[Chap.\,2 \S 9\,(9.4)]{Sz}).
Let us observe that the extension (\ref{extension_d})
splits 
if and only if $n(n-1)\in (\Z/d\Z)^\times$.
In fact, a system of lifts of the generators 
$\bar\sigma_i\in \mathcal{B}_n$
($i=1,\dots,n-1$)
can be written in the form of images of
$\sigma_i \zeta_i^{a_i} \in B_n$ in $B_n/C_n^d$
($a_i\in\Z$). It is easy to see that they
satisfy the braid relations modulo $C_n^d$
if and only if $a_1\equiv \cdots\equiv a_{n-1}$
and $1+n \sum_i a_i\equiv 0$ in $\Z/d\Z$
(cf.\,(\ref{phinu_zeta_n})).
This condition to be held by a collection $\{a_i\}_i$
is equivalent to 
$n(n-1)\in (\Z/d\Z)^\times$ as desired.
Let $p$ be a prime dividing $n(n-1)$ and consider 
$[\mu_p]\in H^2(\mathcal{B}_n,\Z/p\Z)$.
It follows from the above observation that $[\mu_p]\ne 0$.
Since the restriction map
$\mathrm{res}_p: H^1(B_n,\Z/p\Z)
\to H^1(C_n,\Z/p\Z)$ is trivial
under the assumption $p \mid n(n-1)$,
the transgression $\mathrm{tg_p}$ injects
$H^1(C_n,\Z/p\Z)\cong\Z/p\Z$
into $H^2(\mathcal{B}_n,\Z/p\Z)$
whose image is generated by $[\mu_p]\ne 0$.
But for any multiple $N$ of $n(n-1)$, 
the class $[\mu_N]\in \im(\mathrm{tg}_N) (\cong \Z/n(n-1)\Z)
\subset H^2(\mathcal{B}_n,C_n/C_n^N)$
is mapped to $[\mu_p]\in H^2(\mathcal{B}_n,C_n/C_n^p)$
via the reduction of central extensions induced from
the surjective homomorphism $B_n/C_n^N\epi B_n/C_n^p$
in virtue of (\ref{suzuki-id}).
In particular, the reduction map  
$\im(\mathrm{tg}_N)\to \im(\mathrm{tg}_p)$
is given simply by the mod $p$ surjection between 
the cyclic groups:
\begin{equation}
\label{diag-suzuki}
\begin{matrix}
\quad [\mu_N]& \in\im(\mathrm{tg}_N) & (\cong\Z/n(n-1)\Z) 
&\subset H^2(\mathcal{B}_n,C_n/C_n^N) &\qquad
\\
\quad\downarrow & & \qquad\downarrow \mathrm{ mod}\ p \quad &
\downarrow &
\\
0\ne [\mu_p] &\in\im(\mathrm{tg}_p) & (\cong\Z/p\Z) \quad
&\subset H^2(\mathcal{B}_n,C_n/C_n^p). & 
\end{matrix}
\end{equation}
Since the class $[\mu]\in \co^2(\mbnh, \cnh)$ is
the common limit of those $[\mu_N]$,
it follows that $[\mu]$ generates the $p$-primary
component of the cyclic group $\co^2(\mbnh, \cnh) 
\cong \Z/n(n-1)\Z$ for every prime $p\mid n(n-1)$, hence
gives a generator of it.
\end{proof}


\end{document}